\documentclass[a4paper]{article}
\usepackage{amsthm,amsfonts,amsmath,amssymb,units}
\usepackage[abbrev,nobysame]{amsrefs}
\usepackage[cp1251]{inputenc}
\usepackage[english]{babel}
\usepackage[final]{graphicx}
\usepackage{setspace}
\usepackage[12pt]{extsizes}
\begin{document}
\newtheorem{teorema}{Theorem}
\newtheorem{lemma}[teorema]{Lemma}
\newtheorem{utv}[teorema]{Proposition}
\newtheorem{svoistvo}[teorema]{Property}
\newtheorem{sled}[teorema]{Corollary}
\newtheorem{con}[teorema]{Conjecture}
\newtheorem{zam}[teorema]{Remark}
\newtheorem{const}[teorema]{Construction}
\newtheorem{quest}[teorema]{Question}
\newtheorem{problem}[teorema]{Problem}

\author{A. A. Taranenko\thanks{Sobolev Institute of Mathematics, Novosibirsk, Russia. 
 \texttt{taa@math.nsc.ru}}}
\title{Multidimensional threshold matrices and extremal matrices of order $2$}
\date{}

\maketitle

\begin{abstract}
The paper is devoted to multidimensional $(0,1)$-matrices extremal with respect to containing a polydiagonal (a fractional generalization of a diagonal). Every extremal matrix is a threshold matrix, i.e., an entry belongs to its support whenever a weighted sum of incident hyperplanes exceeds a given threshold.  

Firstly, we prove that nonequivalent threshold matrices  have different distributions of ones in hyperplanes.  Next, we establish that extremal matrices of order $2$ are exactly selfdual threshold Boolean functions. Using this fact, we find the asymptotics of the number of extremal matrices of order $2$ and provide counterexamples to several conjectures on extremal matrices. Finally, we describe extremal matrices of order $2$ with a small diversity of hyperplanes. 

 
\end{abstract}

\section*{Introduction}

A \textit{$d$-dimensional $(0,1)$-matrix $A$ of order $n$} is an array $(a_\alpha)_{\alpha \in I^d_n}$, $a_\alpha \in \{ 0 ,1\}$, with the index set $I_n^d= \left\{ (\alpha_1, \ldots , \alpha_d):\alpha_i \in \left\{0, \ldots, n-1 \right\}\right\}$.
The \textit{support} $supp(A)$ of a matrix $A$ is the set of indices $\alpha$ of all nonzero entries $a_\alpha$.

A \textit{hyperplane} in a $d$-dimensional matrix $A$ is a maximal $(d-1)$-dimensional submatrix of $A$ obtained by fixing one of the coordinates. Denote by $\Gamma_{i,j}$ the $j$-th hyperplane of \textit{direction} $i$ in the matrix $A$: $\Gamma_{i,j} = \{ \alpha \, | \,  \alpha_i = j \} $.  We will say that hyperplanes are \textit{parallel} if they have the same direction. 

A \textit{polydiagonal} in a $(0,1)$-matrix $A$  is an assignment of nonnegative weights to entries from the support of $A$ such that the sum of the weights over each hyperplane is equal to 1. A polydiagonal can be considered as a fractional  diagonal. Here a diagonal in a multidimensional matrix is a collection of $n$ unity entries such that each hyperplane contains exactly one of them. 

A $d$-dimensional  $(0,1)$-matrix $A$ is called \textit{extremal} if $A$ has no polydiagonals, and after adding any entry to the support of $A$  the resulting matrix has a polydiagonal.

The study of multidimensional extremal matrices was initiated in~\cite{my.KHEthm}. The main motivation was  the problem of multidimensional generalization of the K\"onig--Hall--Egerv\'ary theorem. Indeed, the K\"onig--Hall--Egerv\'ary theorem  is equivalent to a characterization of $2$-dimensional extremal matrices: all $2$-dimensional $(0,1)$-matrices extremal for polydiagonals are exactly the matrices in which the zero entries form an $s \times t$ submatrix with $s + t = n + 1$.

A polydiagonal in multidimensional matrices is also equivalent to a fractional perfect matching in a $d$-partite $d$-uniform hypergraph with parts of equal sizes. For sufficient conditions on a hypergraph to contain a fractional perfect matching or a  matching of a given weight see, for example, papers~\cite{AhaGeoSpr.pminhyper, KeevMyc.geommatch, RodlRus.hyperDirac} and references therein. 

In this paper, we study extremal  multidimensional matrices as threshold matrices and pay special attention to extremal matrices of order $2$. The structure of the paper is following. 

In Section~\ref{sec-prelim}, we give required definitions on extremal and threshold matrices and recall some results and conjectures from~\cite{my.KHEthm}. 

Section~\ref{threshsec} is the main part of the paper. Firstly, we prove that nonequivalent threshold matrices have different  profiles of the distribution of zeros and ones in hyperplanes. Using this fact, we obtain an upper bound on the number of threshold and extremal matrices of a given order and dimension. Next we show that the set of extremal matrices of order $2$ coincides with the set of selfdual threshold Boolean functions. Thanks to this correspondence, we find counterexamples to Conjectures~3.4 and~3.9  from paper~\cite{my.KHEthm}. Moreover, it gives us the asymptotics of the number of extremal matrices of order $2$. At the end of the section, we consider problems on the asymptotic behaviour of the numbers of extremal and threshold matrices of large order or dimension. 

Finally, in Section~\ref{diversec} we provide a characterization of extremal matrices of order $2$ with small diversity of hyperplanes, i.e., matrices that have no more than two possible fillings of a pair of hyperplanes of each direction.

\section{Definitions and preliminaries} \label{sec-prelim}

We will say that two multidimensional matrices are \textit{equivalent} if one can be turned to the other by permutations of parallel hyperplanes and by permutations of directions of hyperplanes (matrix transpositions). In what follows, we always consider multidimensional matrices and their properties up to equivalence.

Let $A$ be a $d$-dimensional $(0,1)$-matrix of order $n$.
A \textit{polyplex} $K$ of \textit{weight} $W$ in  $A$  is an assignment of nonnegative weights $k_\alpha$ to the entries $\alpha$ from the support of $A$ such that the sum of weights over each hyperplane is not greater than $1$ and the sum of all weights equals $W$. 
An \textit{optimal} polyplex in a matrix $A$ is a polyplex of the maximum weight contained in $A$. Note that  a polyplex  of weight $n$ (maximal possible weight for matrices of order $n$)  is a polydiagonal.

Let the \textit{deficiency} $\delta$ of an extremal matrix $A$ of order $n$ be the difference between $n$ and the weight of its optimal polyplex. In \cite[Proposition 3.6]{my.KHEthm} it was proved that the deficiency $\delta$ of every extremal matrix is a rational number and $0 < \delta \leq 1$.

A \textit{hyperplane cover} of a $d$-dimensional $(0,1)$-matrix $A$ of order $n$ is a $(d \times n)$-table $\Lambda = (\lambda_{i,j})$  assigning nonnegative weights to all hyperplanes of $A$ in such a way that each unity entry of $A$ is covered with weight not less than $1$. In other words, we demand  that for each $\alpha \in supp (A)$ it holds
$\sum\limits_{\Gamma_{i,j}\ni \alpha} \lambda_{i,j} \geq 1$.
The \textit{weight} of a hyperplane cover $\Lambda$ is the sum of all its entries $\lambda_{i,j}$. The set of all hyperplane covers of a given weight is a convex polytope.

We call a hyperplane cover $\Lambda$ of a matrix $A$ \textit{optimal} if it has the minimum weight among all hyperplane covers of $A$. It is not hard to show (see~\cite{my.KHEthm}) that for every matrix there is an optimal hyperplane cover $\Lambda$ with rational weights $\lambda_{i,j}$.

Note that every $(0,1)$-matrix $A$ of order $n$ has a hyperplane cover of weight $n$ because we can always cover the support of $A$ by $n$ hyperplanes of the same direction having weight $1$.

By~\cite[Theorem 2.1]{my.KHEthm},  the weight of an optimal polyplex in an extremal matrix is equal to the weight of its optimal hyperplane cover. For future reasons, we also need the following result.

 \begin{teorema}~\cite[Theorem 2.2]{my.KHEthm} \label{dualityutv}
Assume that $A$ is a $d$-dimensional $(0,1)$-matrix of order $n$, $\Lambda = (\lambda_{i,j})$ is its optimal hyperplane cover, and $K = (k_\alpha)$ is an optimal polyplex  in $A$.
\begin{enumerate}
\item If entry $k_{\alpha} > 0$ then index $\alpha$ is covered with weight $1$ by $\Lambda$; if some index $\alpha$ is covered by $\Lambda$ with not unity weight then $k_\alpha = 0$ for all optimal polyplexes in $A$.
\item If $\lambda_{i,j} > 0$ then the sum of entries of $K$ in the hyperplane $\Gamma_{i,j}$ equals $1$; if in a hyperplane $\Gamma_{i,j}$ the sum of entries of $K$ is not equal to $1$ then $\lambda_{i,j} = 0$ for all optimal hyperplane covers of $A$.
\end{enumerate} 
 \end{teorema}

Recall that if $A$ is an extremal matrix of order $n$, then the weight of an optimal polyplex in $A$ is equal to $n - \delta$ for some $\delta >0$. From Theorem~\ref{dualityutv} it follows that if $A$ is a $d$-dimensional extremal matrix of order $2$, then in an optimal hyperplane cover $\Lambda$ of $A$ one of two parallel hyperplanes $\Gamma_{i,1}$ and $\Gamma_{i,2}$ always has a zero weight. So, for extremal matrices $A$ of order $2$,  an optimal hyperplane cover $\Lambda$  can be written as tuple $(\lambda_1, \ldots, \lambda_d)$  of \textit{essential weights} $\lambda_i$ of hyperplanes for each direction $i$. From~\cite[Construction 6.1]{my.KHEthm} we see that a hyperplane cover $\Lambda' = (\lambda_1, \ldots, \lambda_d, 0, \ldots, 0)$ corresponds to an extremal matrix if and only if $\Lambda = (\lambda_1, \ldots, \lambda_d)$ corresponds to an extremal matrix. So in what follows we assume that all essential weights $\lambda_i$ are positive.
 
Next, we define threshold matrices that are natural generalizations of threshold Boolean functions  to greater orders. 

We will say that a $d$-dimensional $(0,1)$-matrix $A = (a_{\alpha})_{\alpha\in I_n^d}$ of order $n$ is a \textit{threshold matrix} and write $A = A(\Lambda)$ if there is a $(d \times n)$-table  $\Lambda$  of nonnegative weights $\lambda_{i,j}$ such that
$$a_{\alpha} = 1 \Leftrightarrow \sum\limits_{\Gamma_{i,j}\ni \alpha} \lambda_{i,j} \geq 1.$$

There is an another way to introduce threshold matrices. To every index $\alpha$ of a $d$-dimensional matrix of order $n$  we put into correspondence  a $(0,1)$-table $V^{\alpha}$  of size $d \times n$, in which $v^{\alpha}_{i,j} = 1$ if and only if $\alpha_i = j$.   Then the matrix $A$ is threshold if there exists a hyperplane in $\mathbb{R}^{dn}$ defined by vector $L$ that separates any $V^{\alpha} \in \mathbb{R}^{dn}$, $\alpha \in supp(A)$, from all $V^{\beta}$, $\beta \not\in supp(A)$:
 $$a_{\alpha} = 1 \Leftrightarrow \sum\limits_{i,j} \ell_{i,j} v^{\alpha}_{i,j}  \geq 1.$$

 \begin{teorema}~\cite[Theorem 4.1]{my.KHEthm} \label{coverextr}
 Every extremal matrix $A$ is equivalent to a threshold matrix. Moreover, if $\Lambda$ is an optimal hyperplane cover of an extremal matrix $A$, then $A = A(\Lambda)$.
 \end{teorema}

Most results of paper~\cite{my.KHEthm} were motivated by the following conjectures.

\begin{con}~\cite[Conjecture 3.4]{my.KHEthm} \label{uniquecon}
Every extremal matrix has a unique optimal cover.
\end{con}

\begin{con}~\cite[Conjecture 3.9]{my.KHEthm}  \label{defmulti}
If $\Lambda$ is an optimal hyperplane cover of an extremal matrix of deficiency $\delta$, then all entries $\lambda_{i,j}$ are integer multiples of $\delta$.
\end{con}

\begin{con}~\cite[Conjecture 3.7]{my.KHEthm}  \label{defcon}
The deficiency $\delta$ of any extremal matrix is equal to $\nicefrac{1}{m}$ for some $m \in \mathbb{N}$.
\end{con}

In this paper we provide counterexamples to Conjectures~\ref{uniquecon} and~\ref{defmulti} in the class of extremal matrices of order $2$. Meanwhile, in Section~\ref{diversec} we show that all of these conjectures are true for extremal matrices of small diversity of hyperplanes.

\section{Threshold functions and bounds on the number of extremal matrices} \label{threshsec}

Suppose that $A$ is a $d$-dimensional $(0,1)$-matrix of order $n$.  Let the \textit{rate} $r_{i,j}$ of a hyperplane $\Gamma_{i,j}$ be the sum of all entries in $\Gamma_{i,j}$.   The \textit{profile} $R = R(A)$ of the matrix $A$ is the multiset of all rates $r_{i,j}$ of hyperplanes in $A$.

Let us  prove that every multidimensional threshold matrix is uniquely defined by its profile. The similar fact for threshold matrices of order $2$ (threshold Boolean functions) was obtained in~\cite{chow.charthresh}.

\begin{teorema}\label{weightspectr}
If $A$ and $B$  are nonequivalent $d$-dimensional threshold matrices of order $n$, then they have different profiles.
\end{teorema}

\begin{proof}
Assume that  nonequivalent $d$-dimensional threshold matrices $A$ and $B$ of order $n$ have the same profile. Let us order hyperplanes and their directions in $A$ and $B$ so that the rates of the corresponding hyperplanes coincide.   Note that sets $supp (A) \setminus supp (B)$ and $supp (B) \setminus supp (A)$ are both nonempty, because matrices $A$ and $B$ are nonequivalent.

Since they have the same profile, for the tables $V^\alpha$ of indices $\alpha$  we get
$$\sum\limits_{\alpha \in supp(A) \setminus supp(B)} V^\alpha = \sum\limits_{\beta \in supp(B) \setminus supp(A)} V^\beta, $$
where summation of tables $V^\alpha$ is coordinate-wise.
It means that there are subsets  $\mathcal{V} \subseteq \{V^\alpha| \alpha \in supp (A) \}$ and $\mathcal{V}' \subseteq   \{V^\beta| \beta \not\in supp (A) \}$ such that  the convex hulls of $\mathcal{V}$ and $\mathcal{V}'$ intersect. Consequently, there are no hyperplanes in $\mathbb{R}^{dn}$ separating $\mathcal{V}$ and $\mathcal{V}'$ and the matrix $A$ is not threshold.
\end{proof}

\subsection{Extremal matrices of order $2$ and selfdual threshold Boolean fun\-cti\-ons}

A function $f : \{ 0,1\}^d \rightarrow \{ 0,1\}$ is called a \textit{Boolean function} on $d$ variables. Every $d$-dimensional $(0,1)$-matrix of order $2$ can be considered as a table of values of a Boolean function $f =  f(x_1, \ldots, x_d)$.

A Boolean function $f$ on $d$ variables is called \textit{threshold} if there are weights $\mu_i$ and a threshold $T$ such that $f(x_1, \ldots, x_d) = 1 \Leftrightarrow \sum\limits_{i=1}^d \mu_i x_i \geq T$. Note that threshold Boolean functions $f$ are equivalent to threshold $(0,1)$-matrices $A(\Lambda)$: it is sufficient to take a  threshold $T = 1$ and define the weights $\lambda_{i,j}$ of the threshold matrix $\Lambda$ as $\lambda_{i, 1} = \mu_i$, $\lambda_{i, 2} = 0$ for all $i = 1, \ldots, d$. 

For $a \in \{0,1\} $ and $x \in \{ 0,1\}^d$,  let $\overline{a} = a \oplus 1$ and $\overline{x} = (\overline{x_1}, \ldots, \overline{x_d})$. We will say that $x$ and $\overline{x}$ are \textit{antipodal}.  
A Boolean function $f$ is said to be \textit{selfdual} if $f (x) = \overline{f} (\overline{x})$.  
In other words, the table of every selfdual Boolean function is a $(0,1)$-matrix of order $2$ such that for every pair of indices $\alpha$ and $\overline{\alpha}$ exactly one index belongs to the support of $A$.

Let us establish the equivalence between extremal matrices of order $2$ and selfdual threshold Boolean functions.

\begin{teorema} \label{extrselfthresh}
Extremal matrices of order $2$ are exactly selfdual threshold Boolean functions.
\end{teorema}

\begin{proof}
In~\cite[Theorem 8.1]{my.KHEthm} it is proved that every extremal matrix of order $2$  is a  selfdual Boolean function (an antipodal matrix).  By Theorem~\ref{coverextr}, every extremal matrix is threshold. So the set of extremal  matrices of order $2$ is contained in the set of selfdual threshold Boolean functions. Let us prove the converse statement.

Assume that $f$ is a selfdual threshold Boolean function and a matrix  $A$ of order $2$ is the table of values of $f$. Since $f$ is selfdual, any new entry in the support of $A$ produces a diagonal and, consequently, a polydiagonal. 

Suppose that the matrix $A$ contains a polydiagonal. Then every optimal hyper\-plane  cover $\Lambda$ of $A$ has weight $2$.   By Theorem~\ref{dualityutv}, there is an index $\alpha$ covered with weight $1$ by $\Lambda$. Consequently, the antipodal index $\overline{\alpha}$ is also covered  with weight $1$ by $\Lambda$. Since $A$  is a threshold matrix, both $\alpha$ and $\overline{\alpha}$ belong to the support of $A$: a contradiction with selfduality of $f$. Thus, $A$ has no polydiagonals and $A$ is an extremal matrix.
\end{proof}

Selfdual threshold Boolean functions were studied in a number of papers. 
For example, in~\cite{MuTsuBau.Thresh8var} there was an enumeration of all selfdual threshold functions on small number of variables. In particular, \cite{MuTsuBau.Thresh8var} discovered $12$ selfdual threshold functions on $9$ variables such that the set of optimal hyperplane covers for the corresponding extremal matrices is a $1$-dimensional polytope (edge).     Consequently, these  extremal matrices  give counterexamples to Conjectures~\ref{uniquecon} and~\ref{defmulti}. 

 Here are the vertices of the polytopes of essential weights for the above extremal matrices, where the second vertex of a polytope  is obtained by interchanging bold coefficients: 
\begin{gather*}
\nicefrac{1}{25} \cdot (13,7,6,6,4,4,4,\textbf{3}, \textbf{2}); \\
\nicefrac{1}{30} \cdot (17,9,8,\textbf{7},\textbf{6},5,3,2,2 ); \\
\nicefrac{1}{28} \cdot (13,9,7,7,6,4,4,\textbf{3}, \textbf{2}); \\
\nicefrac{1}{27} \cdot (14,9,\textbf{7},\textbf{6},5,5,3,2,2 ); \\
\nicefrac{1}{33} \cdot (17,12,8,8,\textbf{7},\textbf{6},3,2,2 ); \\
\nicefrac{1}{24} \cdot (11,9,6,6,4,4,4,\textbf{2}, \textbf{1}); \\
\nicefrac{1}{28} \cdot (13,11,7,7,5,5,4,\textbf{2}, \textbf{1}); \\
\nicefrac{1}{28} \cdot (13,11,8,6,6,4,4,\textbf{2}, \textbf{1}); \\
\nicefrac{1}{32} \cdot (15,13,9,7,7,5,4,\textbf{2}, \textbf{1}); \\
\nicefrac{1}{33} \cdot (13,11,10,8,6,6,\textbf{5}, \textbf{4}, 2); \\
\nicefrac{1}{34} \cdot (16,14,11,9,6,4,4,\textbf{2}, \textbf{1}); \\
\nicefrac{1}{38} \cdot (18,16,12,10,7,5,4,\textbf{2}, \textbf{1}). \\
\end{gather*}

\subsection{Bounds on the numbers of threshold and extremal matrices}

Unique\-ness of an extremal matrix for a given profile allows us to estimate the number of threshold matrices from  above.

\begin{teorema} \label{thresupbound}
The number of nonequivalent $d$-dimensional threshold matrices of order $n$ is not greater than $n^{nd(d-1)} +1$.
\end{teorema}

\begin{proof}
Let $R(A) = (r_{i,j})$, $i = 1, \ldots, d$, $j = 1, \ldots, n$, be the profile of  a $d$-dimensional threshold matrix $A$ of order $n$. It is clear that every rate $r_{i,j}$  is between $0$ and $n^{d-1}$. Note that if some hyperplane of an extremal matrix has rate $0$ (contains no $1$s), then all parallel hyperplanes have rate $n^{d-1}$ (see~\cite{my.KHEthm} for more details).   By Theorem~\ref{weightspectr}, there are no two nonequivalent threshold matrices with the same profile. So, their number does not exceed $(n^{d-1})^{nd}  +1= n^{nd(d-1)} +1.$
\end{proof}

Since, by  Theorem~\ref{coverextr}, every multidimensional  extremal matrix is threshold, we also have an upper bound on the number of extremal matrices.

\begin{sled} \label{extrupbound}
The number of nonequivalent $d$-dimensional extremal matrices of order $n$ is not greater than $n^{nd(d-1)} +1$.
\end{sled}

Using the correspondence between extremal matrices of order $2$ and selfdual threshold Boolean functions, one can obtain a better bound on the number of such matrices. Upper bounds on the number of threshold Boolean functions have been known since 1960s (see, for example,~\cite{winder.threshold}). For the sake of completeness, we show  how the bound from Theorem~\ref{thresupbound} can be improved for selfdual threshold Boolean functions.  

Firstly, we need one additional property of profiles of such matrices.

\begin{utv} \label{rateparity}
Let $A$ be a $d$-dimensional extremal matrix of order $2$ with the profile $R(A) $. Then all rates $r_{i,j}$ have the same parity.
\end{utv}

\begin{proof}
From the definitions we have that  every selfdual threshold Boolean function $f$ on $d$ variables can be defined by a hyperplane $H$ in $\mathbb{R}^d$: $H$ goes through the center of the Boolean hypercube $[0,1]^d$ and separates the vertices of  $[0,1]^d$ belonging to the support of $f$ from the others.  Rotations of the hyperplane $H$ with respect to the center produce all selfdual threshold Boolean functions on $d$ variables.

Note that every time  the hyperplane $H$ goes through some vertex $\alpha \in [0,1]^d$ which adds it to the support of $f$, $H$ deletes the antipodal vertex $\overline{\alpha}$ from the support of $f$. The described operation changes  the rates $r_{i,j}$ of all hyperplanes by $\pm 1$. It only remains to note that there is an extremal matrix $A$ with profile $R(A)$, where all $r_{i,j}$ have the same parity. For example,  $A$ is a $d$-dimensional matrix such that the support of $A$ is some hyperplane $\Gamma$: in this case all $r_{i,j} \in \{ 0, 2^{d-2}, 2^{d-1}\}$. 
\end{proof}

\begin{teorema}
The number of nonequivalent $d$-dimensional extremal matrices of order $2$ is not greater than $2^{d(d-3)}$.
\end{teorema}

\begin{proof}
By Theorems~\ref{coverextr} and~\ref{weightspectr},  it is sufficient to estimate the number of different profiles $R(A)$ of $d$-dimensional extremal matrices $A$ of order $2$. From Theorem~\ref{extrselfthresh} we also have that extremal matrices of order $2$ are equivalent to selfdual threshold  Boolean functions. 

Note that the size of the support  of every selfdual Boolean function $f$ on $d$ variables (and so a $d$-dimensional matrix of order $2$) is $2^{d-1}$. To find the profile of a $d$-dimensional extremal matrix of order $2$, it is sufficient to know the rate $r_{i,j}$ of one hyperplane for each of $d$ directions (e.g., the smallest one) that takes only $2^{d-2}$ values. 

By Proposition~\ref{rateparity}, rates of all hyperplanes for an extremal matrix of order $2$ have the same parity.  That reduces the number of possibilities for rates $r_{i,j}$ to $2^{d-3}$. Therefore, there are no more than $(2^{d-3})^d$ different profiles $R(A)$ of $d$-dimensional extremal matrices $A$ of order $2$.
\end{proof}

It is not easy to find rich families of threshold and selfdual functions.
One  of the first asymptotic lower bound on threshold  Boolean functions was given in~\cite{smith.boundthresh}. Some constructions of selfdual threshold Boolean functions are proposed in~\cite{muroga.genselfthresh, muroga.gensymthresh}. At last, in~\cite{zuev.asthresh} it is found the asymptotics of  logarithms of the numbers of selfdual threshold Boolean functions and threshold Boolean functions.

\begin{teorema}[\cite{zuev.asthresh}]
Let $M(d)$ be the number of nonequivalent selfdual threshold Boolean functions on $d$ variables. Then $\log M(d) = d^2(1 + o(1)) $ as $d \rightarrow \infty$.
\end{teorema}

As an immediate corollary of Theorem~\ref{extrselfthresh}, we have the asymptotics of the number of extremal matrices of order $2$.

\begin{teorema}
Let $M(d,2)$ be the number of nonequivalent $d$-dimensional extremal matrices of order $2$. Then $\log M(d,2) = d^2(1 + o(1)) $ as $d \rightarrow \infty$.
\end{teorema}

Good lower bounds on the number of $d$-dimensional threshold  and extremal matrices of order $n$  are unknown. We believe that for fixed order they should be close to the upper bounds from Theorems~\ref{thresupbound} and \ref{extrupbound}.

\begin{con} \label{asympcon}
 Let $M(d,n)$ be the number of nonequivalent $d$-dimensional extremal matrices of order $n$ and $T(d,n)$ be the number of nonequivalent $d$-dimensional threshold matrices of order $n$.
 Given $n$, there is a constant $c = c(n)$ such that  $\log M(d,n) = c d^2  (1 + o(1))$  and $  \log T(d,n) =   c d^2  (1 + o(1))$ as $d \rightarrow \infty$.
\end{con}

When the dimension of matrices is fixed,  the number of  extremal matrices of order $n$ could be  much less than  the number of threshold matrices of the same order.

Indeed,  the classical K\"onig--Hall--Egerv\'ary theorem gives a characterization of $2$-dimensional extremal matrices order $n$. Every such extremal matrix is defined by the sizes $s$ and $t$ of its maximal zero submatrix which satisfy $s+t = n+1$. It implies that there are exactly $\lfloor \frac{n+1}{2} \rfloor$ nonequivalent $2$-dimensional extremal matrices of order $n$.

On the other hand, $2$-dimensional threshold matrices of order $n$ are equivalent to  stepped matrices. We will say that a $2$-dimensional $(0,1)$-matrix $A$ is \textit{stepped} if for all $i$ and $j$, $i > j$, the support of the $i$-th row (column) of $A$ is contained in the support of the $j$-th row (column).

\begin{utv}
A $2$-dimensional $(0,1)$-matrix $A$ of order $n$ is threshold if and only if it is equivalent to some stepped matrix.
\end{utv}

\begin{proof}
$\Rightarrow$: Let $A$ be a threshold matrix defined by a hyperplane cover $\Lambda = (\lambda_{i,j})$.  If $\Lambda'$ is a hyperplane cover in which  the weights of rows and columns of $\Lambda$ are arranged in descending order, then $A(\Lambda')$ is a stepped matrix equivalent to $A$.

$\Leftarrow$: Let $A$ be a stepped matrix of order $n$.  We say that an index $(i,j)$ from the support of $A$ is an  outer index if  $a_{i+1, j}$ and $a_{i, j+1}$ are equal to $0$ (we assume that $a_{i, n+1} = a_{n+1, j} = 0$ for all $i$ and $j$).  Note that if some row (column) does not contain an outer index, then it is equal to another row (column) that contains one. 

Suppose that there are exactly $k$ outer indices $(i_1, j_1), \ldots, (i_m, j_m)$ in the matrix $A$, $i_l < i_{l+1}$, $j_{l} > j_{l+1}$ for all $l = 1, \ldots, m-1$.   For each outer index $(i_l, j_l)$,  define $\lambda_{1,i} = \frac{k - l+1}{k+1}$ and $\lambda_{2,j} = \frac{l}{k+1}$ for all rows $i$ and columns $j$ that coincide with the $i_l$-th row or the $j_l$-th column. It can be checked that $A = A(\Lambda)$ for this table of weights $\Lambda$.
\end{proof}

Stepped matrices of order $n$ are in  one-to-one correspondence with staircase walks from $(0,0)$ to $(n,n)$.  So the number of nonequivalent $2$-dimensional threshold matrices of order $n$ can be estimated as ${2n  \choose n} \sim 2^{2n}$ that is much bigger than the number of extremal matrices of order $n$. 

\begin{problem}
Let $d$ be given. 
\begin{enumerate}
\item Find the asymptotics of the number of nonequivalent $d$-dimensional threshold matrices of order $n$ as $n \rightarrow \infty$.
\item Find the asymptotics of the number of nonequivalent $d$-dimensional extremal matrices of order $n$ as $n \rightarrow \infty$.
\end{enumerate}
\end{problem}

\section{Extremal matrices of order $2$ and small diversity} \label{diversec}

In this section we consider only extremal matrices of order $2$. 
Recall that every $d$-dimensional extremal matrix $A$ of order $2$ is defined by  essential weights of its optimal cover $\Lambda = (\lambda_1, \ldots, \lambda_d)$.

If  an optimal hyperplane cover $\Lambda$ of a $d$-dimensional extremal matrix $A$ consists of $m$ different positive real numbers $x_1, \ldots, x_m$, then we will say that $A$ is a matrix of \textit{diversity} $m$. In other words, a matrix $A$ has a diversity $m$ if the set $\Gamma_{1,1}, \ldots, \Gamma_{d,1}$ of hyperplanes of each direction contains exactly $m$ nonequivalent hyperplanes.

Using results from~\cite{my.KHEthm},  it is not hard to describe extremal matrices of order $2$ and diversity $1$.

\begin{utv}
A hyperplane cover $\Lambda = (\lambda_1, \ldots, \lambda_d)$ defines a $d$-dimensional extremal matrix of order $2$ and diversity $1$ if and only if $d$ is odd and all essential weights $\lambda_{i} = \frac{2}{d+1}$.
\end{utv}

\begin{proof}
By~\cite[Theorem 6.6]{my.KHEthm}, $d$-dimensional matrices of diversity $1$ have deficiency $\delta = \nicefrac{1}{k}$ for some $k \in \mathbb{N}$, and their optimal hyperplane covers $\Lambda$ have nonzero weights $\lambda_{i,j}  $ equal to $\nicefrac{1}{k}$. If an extremal matrix $A$ of diversity $1$ has dimension $d$ and order $2$, then for essential weights $\lambda$ it holds
$$\frac{d}{k} =  d \lambda = 2 - \delta  = \frac{2k-1}{k}.$$
From here we deduce that $d = 2k - 1$  and $\lambda = \nicefrac{1}{k} = \frac{2}{d+1}$.
\end{proof}

Let us characterize extremal matrices of diversity $2$. 
 
In what follows we assume that $\Lambda$ is a $d$-tuple $ (x, \ldots, x, y, \ldots, y)$, where $1 > x > y > 0$, $x$ and $y$ appear $t_x$ and $t_y$ times, respectively, $t_x +t_y = d$.

\begin{teorema} \label{div2char}
A tuple $\Lambda =  (x, \ldots, x, y, \ldots, y)$ with $x = \frac{p}{q}$, $y = \frac{s}{q}$  is an optimal hyperplane cover of an extremal matrix $A = A(\Lambda)$ of order $2$  if and only if  $\gcd(p,s) = 1$,  $t_x p + t_y s = 2q-1$, and there are integer $r_x, r_y, \ell_x, \ell_y $ satisfying the following conditions:
\begin{itemize}
\item $0 \leq r_x  < \ell_x \leq t_x$ and $0 \leq \ell_y <  r_y \leq t_y$;
\item  $\ell_x = r_x + is$ and $\ell_y = r_y - ip$ for some $i \in \mathbb{N}$;
\item $r_x < s$ or $r_y + p > t_y$, and $\ell_x + s > t_x$ or $\ell_y < p$;
\item $r_x p  +  r_y s = \ell_x p + \ell_y s = q$.
\end{itemize}
 In particular, we have that the deficiency of the matrix $A$ is $\delta = \frac{1}{q}$. 
 \end{teorema}

\begin{proof}
$\Leftarrow$:   Since both $x$ and $y$ are integer multiples of $q$ and $t_x p + t_y s = 2q-1$, an index $\alpha$ is covered by $\Lambda$ with weight less than $1$  if and only if the antipodal index $\overline{\alpha}$ is covered with weight at least $1$. So $A = A(\Lambda)$ is a table of values of a selfdual Boolean function, and, by Theorem~\ref{extrselfthresh}, $A$ is an extremal matrix.   It remains to show that $\Lambda$ is an optimal hyperplane cover of $A$, i.e., that the matrix $A$ contains a polyplex of weight $2 - 1/q$. 

Let $\mathcal{I}$ be the set of indices $\alpha$ such that $\Lambda$ covers $\alpha$ exactly $r_x$ times by hyperplanes of weight $x$, $r_y$ times by hyperplanes of weight $y$, and $d - r_x - r_y$ times by hyperplanes having weight $0$. Note that $|\mathcal{I}| = {t_x \choose r_x} {t_y \choose r_y} $ and condition $r_x p  +  r_y s = q$ implies that $\mathcal{I} \subset supp(A)$. Similarly, we define the set of indices $\mathcal{J}$ such that $\Lambda$ covers each $\beta \in \mathcal{J}$ exactly $\ell_x$ times by hyperplanes of weight $x$, $\ell_y$ times by hyperplanes of weight $y$, and $d - \ell_x - \ell_y$ times by hyperplanes of weight $0$. By the conditions of the theorem, we see that $\mathcal{J}$ is also a subset in the support of $A$ and $|\mathcal{J}| = {t_x \choose \ell_x} {t_y \choose \ell_y} $.

We look for an optimal polyplex $K$ of weight $2 - 1/q$ in the matrix $A$ such that $supp (K) = \mathcal{I} \cup \mathcal{J}$. Using  the symmetry, we assume that for all $\alpha \in \mathcal{I}$ the entries $k_{\alpha}$ of $K$ have the same weight $w_{I}$ and for all $\beta \in \mathcal{J}$ the entries $k_{\beta}$ have  weight $w_{J}$.  From Theorem~\ref{dualityutv} we find the sums of entries of $K$ in each hyperplane,  that gives us the following system on weights $w_I$ and $w_J$:
$$
\left\{
\begin{array}{l}
{t_x -1\choose r_x -1} {t_y \choose r_y} w_I + {t_x -1\choose \ell_x -1} {t_y \choose \ell_y} w_J = 1; \\
{t_x \choose r_x } {t_y -1 \choose r_y -1} w_I + {t_x \choose \ell_x} {t_y -1 \choose \ell_y -1} w_J = 1; \\
{t_x -1\choose r_x } {t_y \choose r_y} w_I + {t_x -1\choose \ell_x } {t_y \choose \ell_y} w_J = 1 - 1/q; \\
{t_x \choose r_x } {t_y -1 \choose r_y } w_I + {t_x \choose \ell_x} {t_y -1 \choose \ell_y} w_J = 1-1/q. \\
\end{array}
\right.
$$

It can be checked that this system is consistent and has the solution

$$
w_I =  {t_x \choose r_x }^{-1} {t_y \choose r_y}^{-1} \frac{\ell_x t_ y - \ell_y t_x }{ r_y \ell_x - r_x \ell_y} = {t_x \choose r_x }^{-1} {t_y \choose r_y}^{-1} \frac{\ell_x t_ y - \ell_y t_x }{iq} ;
$$
$$
w_J = {t_x \choose \ell_x}^{-1} {t_y \choose \ell_y}^{-1} \frac{r_y t_x - r_x t_y}{r_y \ell_x - r_x \ell_y} = {t_x \choose \ell_x}^{-1} {t_y \choose \ell_y}^{-1} \frac{r_y t_x - r_x t_y}{iq},
$$
because  $\ell_x = r_x + is$,  $r_y = \ell_y + ip$, and $r_x p  +  r_y s = q$.

Let us show that the weights $w_I$ and $w_J$ are nonnegative. 
By the conditions, we have that $t_x p + t_y s = 2q-1$ and $r_x p + r_y s = q$. Consequently, 
\begin{equation} \label{eq1}
(2r_x-t_x) p + (2r_y - t_y) s = 1.
\end{equation}
Since  $r_x + s \leq t_x$ and $r_y \geq p$, the inequalities $r_x < s$ or $r_y + p > t_y$ taken with~(\ref{eq1}) imply that $2r_x - t_x \leq 0$ and $2 r_y - t_y \geq 0$. Combining these two inequalities, we deduce  that  $\frac{r_x}{r_y} \leq  \frac{t_x}{t_y}  $, that is equivalent to  $w_J \geq 0$.

Similarly, equations $t_x p + t_y s = 2q-1$ and $\ell_x p + \ell_y s = q$ give that
\begin{equation} \label{eq2}
(2\ell_x -t_x) p + (2\ell_y - t_y) s = 1. 
\end{equation}
Since  $\ell_x \geq s$ and $\ell_y + p  \leq t_y$, the conditions $\ell_x  + s > t_x$ or $\ell_y < p$ taken with~(\ref{eq2}) imply that $2\ell_x - t_x \geq 0$ and $2 \ell_y - t_y \leq 0$. Combining these two inequalities, we deduce  that  $\frac{t_x}{t_y} \leq  \frac{\ell_x}{\ell_y}  $, that is equivalent to $w_I \geq 0$.

$\Rightarrow$: Assume that  $\Lambda =  (x, \ldots, x, y, \ldots, y)$, $x,y \in \mathbb{Q}$, is an optimal hyperplane cover of an extremal matrix $A = A(\Lambda)$ of deficiency $\delta$: $t_x x + t_y y = 2-\delta$. Let $q \in \mathbb{N}$ be the minimal number such that $x = p /q$, $y = s/q$ for some integer $p$ and $s$. 

Since $A$ an extremal matrix, there is an optimal polyplex $K$ in $A$.
By Theorem~\ref{dualityutv}, if $\alpha$ belongs to the support of an optimal polyplex $K$, then $\alpha$ is covered with weight $1$ by $\Lambda$. It means that for some integer $0 \leq r_x \leq t_x$, $0 \leq r_y \leq t_y$ we have $r_x p+r_y s  = q$.  Note that if $\gcd(p,s) = g > 1$, then $q$ is also an integer multiple of $g$ that contradicts the minimality of $q$. 

Since $\gcd(p,s) = 1$, all $\mu_x, \mu_y \in \mathbb{N}$ satisfying $\mu_x p + \mu_y s = q$  have the form $\mu_x = r_x + j s$, $\mu_y = r_y - jp $ for some integer $j$.

Without loss of generality, we assume that $r_x$ is the minimal possible number (and, respectively, $r_y$ is the maximal possible number), for which conditions $0 \leq r_x \leq t_x$, $0 \leq r_y \leq t_y$, and $r_x p+r_y s  = q$  hold. It means that $r_x < s$ or $r_y + p > t_y$, since otherwise for $r'_x = r_x - s$ and $r'_y = r_y + p$ these conditions are also satisfied.

We denote by $\mathcal{I}$ the set of indices $\alpha$ such that $\Lambda$ covers $\alpha$ exactly $r_x$ times by hyperplanes of weight $x$, $r_y$ times by hyperplanes of weight $y$, and $d - r_x - r_y$ times by hyperplanes having weight $0$.

Let us show that  there are other indices $\beta \in supp (A)$ such that for some $i \in \mathbb{N}$  the table $\Lambda$ covers $\beta$ exactly $\ell_x = r_x + is$ times by hyperplanes of weight $x$ and $\ell_y = r_y - ip$ times by hyperplanes of weight $y$.  Indeed, if there are no such indices $\beta$, then the support of every optimal polyplex in $A$ is contained in $\mathcal{I}$. Due to the symmetry, there is an optimal polyplex $K$ in $A$ such that $supp(K) = \mathcal{I}$ and all entries $k_\alpha$ have the same weight $w$. 
 By Theorem~\ref{dualityutv} and the sums of entries of $K$ in each hyperplane,  we have the following equalities on the weight $w$:
 
$$
\left\{ \begin{array}{l}
 {t_x -1 \choose r_x -1} { t_y \choose r_y} w = 1; \\
  {t_x \choose r_x} { t_y - 1\choose r_y-1} w = 1; \\
  {t_x -1 \choose r_x} { t_y \choose r_y} w = 1 - \delta; \\
   {t_x \choose r_x} { t_y -1 \choose r_y} w = 1 - \delta. \\
   \end{array}
   \right.
$$
 
 The consistence of this system is equivalent to $\frac{r_x}{t_x} = \frac{r_y}{t_y}$ and $\frac{r_x}{t_x - r_x} = \frac{r_y}{t_y - r_y} = 1- \delta$. Since $r_x x + r_y y = 1$ and $t_x x + t_y y = 2- \delta$, we deduce that  $\frac{r_x}{t_x} = \frac{r_y}{t_y} = \frac{1}{2- \delta}$.  Note that equations $\frac{r_x}{t_x - r_x} = 1- \delta$ and $\frac{r_x}{t_x} = \frac{1}{2- \delta}$ hold only if $\delta =0$ that contradicts to the extremality of the matrix $A$.
 
Thus for some integer $i > 0$, $\ell_x = r_x + is$ and $\ell_y = r_y - ip$, we have $\ell_x \leq t_x$, $\ell_y \geq 0$ and $\ell_x p + \ell_y s = q$. Without loss of generality, we may assume that $\ell_x$ is the maximal possible number (and, respectively, $\ell_y$ is the minimal possible number) for which these conditions hold. It gives that $\ell_x + s > t_x$ or $\ell_y < p$. 

It remains to show that $t_x p + t_y s = 2q-1$. Assume that $t_x p + t_y s = 2q-h$ for some integer $h$, $1 < h < q$. From the existence of two different pairs $r_x$, $r_y$ and $\ell_x$, $\ell_y$, satisfying $r_x p + r_y s = \ell_x p + \ell_y s = q$, we see that  there are some integers  $g_x$ and $g_y$, $0 \leq g_x \leq t_x$, $0 \leq g_y \leq t_y$ such that $g_x p + g_y s = q - 1$. Then $(t_x - g_x)p + (t_y - g_y)s = q -  h  + 1 < q$. It means that there are two antipodal entries that do not belong to the support of $A$: a contradiction to the extremality of $A$.
\end{proof}

As a corollary of Theorem~\ref{div2char}, we have that the deficiency $\delta$ of every extremal matrix $A$ of order $2$ and diversity $2$ is equal to $1/q$ for some $q \in \mathbb{N}$, $A$ has the unique optimal cover $\Lambda$ and the essential weights of $\Lambda$  are integer multiples of deficiency $\delta$. Thus Conjectures~\ref{uniquecon}, \ref{defmulti}, and \ref{defcon}  hold for all extremal matrices of order $2$ and diversity at most $2$.

\section*{Acknowledgements}

This work was funded by the Russian Science Foundation under grant No 22-21-00202, https:// rscf.ru/project/22-21-00202/.


\bib{AhaGeoSpr.pminhyper}{article}{
  author={Aharoni, Ron},
  author={Georgakopoulos, Agelos},
  author={Spr\"{u}ssel, Philipp},
  title={Perfect matchings in $r$-partite $r$-graphs},
  journal={European J. Combin.},
  volume={30},
  date={2009},
  number={1},
  pages={39--42},
  issn={0195-6698},
}

\bib{chow.charthresh}{article}{
  author={Chow, C. K.},
  title={On the characterization of threshold functions},
  journal={In Proceedings of 2nd Annual Symposium on Switching Circuit Theory and Logical Design (SWCT 1961)},
  date={1961},
  pages={34--38},
}

\bib{KeevMyc.geommatch}{article}{
  author={Keevash, Peter},
  author={Mycroft, Richard},
  title={A geometric theory for hypergraph matching},
  journal={Mem. Amer. Math. Soc.},
  volume={233},
  date={2015},
  number={1098},
  pages={vi+95},
  issn={0065-9266},
  isbn={978-1-4704-0965-4},
}

\bib{muroga.genselfthresh}{article}{
  author={Muroga, Saburo},
  title={Generation of self-dual threshold functions and lower bounds of the number of threshold functions and a maximum weight},
  journal={3rd Annual Symposium on Switching Circuit Theory and Logical Design (SWCT 1962)},
  date={1962},
  pages={169-184},
}

\bib{muroga.gensymthresh}{article}{
  author={Muroga, Saburo},
  title={Generation and Asymmetry of Self-Dual Threshold Functions},
  journal={IEEE Transactions on Electronic Computers},
  volume={EC-14},
  number={2},
  date={1965},
  pages={125-136},
}

\bib{MuTsuBau.Thresh8var}{article}{
  author={Muroga, Saburo},
  author={Tsuboi, Tei\'ichi},
  author={Baugh, Charles R.},
  title={Enumeration of Threshold Functions of Eight Variables},
  journal={IEEE Transactions on Computers},
  volume={C-19},
  number={9},
  date={1970},
  pages={818-825},
}

\bib{RodlRus.hyperDirac}{article}{
  author={R\"{o}dl, Vojtech},
  author={Ruci\'{n}ski, Andrzej},
  title={Dirac-type questions for hypergraphs---a survey (or more problems for Endre to solve)},
  conference={ title={An irregular mind}, },
  book={ series={Bolyai Soc. Math. Stud.}, volume={21}, publisher={J\'{a}nos Bolyai Math. Soc., Budapest}, },
  date={2010},
  pages={561--590},
}

\bib{smith.boundthresh}{article}{
  author={Smith, D. R.},
  title={Bounds on the number of threshold functions},
  journal={IEEE Trans. Electronic Computers},
  volume={EC-15},
  date={1966},
  pages={368--369},
}

\bib{my.KHEthm}{article}{
  author={Taranenko, Anna A.},
  title={On the K\"{o}nig-Hall-Egerv\'{a}ry theorem for multidimensional matrices and multipartite hypergraphs},
  journal={Discrete Math.},
  volume={344},
  date={2021},
  number={8},
  pages={Paper No. 112447, 18},
  issn={0012-365X},
}

\bib{winder.threshold}{article}{
  author={Winder, Robert O.},
  title={Single stage threshold logic},
  conference={ title={2nd Annual Symposium on Switching Circuit Theory and Logical Design (SWCT 1961)}, date={1961}, },
  pages={321-332},
}

\bib{zuev.asthresh}{article}{
  author={Zuev, Yu. A.},
  title={Asymptotics of the logarithm of the number of Boolean threshold functions},
  language={Russian},
  journal={Dokl. Akad. Nauk SSSR},
  volume={306},
  date={1989},
  number={3},
  pages={528--530},
  issn={0002-3264},
  translation={ journal={Soviet Math. Dokl.}, volume={39}, date={1989}, number={3}, pages={512--513}, issn={0197-6788}, },
}



\begin{bibdiv}
    \begin{biblist}[\normalsize]
    \bibselect{biblio}
    \end{biblist}
    \end{bibdiv}
\end{document}